\documentclass[11pt,oneside,english]{amsart}
  \usepackage{amsmath}
  \usepackage{amsfonts}
  \usepackage{amsthm}
  \usepackage{amssymb}
  \usepackage[all]{xy}
  \usepackage{verbatim}  
  \usepackage{longtable}
  \usepackage{stmaryrd}   
  \usepackage{mathrsfs}   
  \usepackage{color}
  \usepackage{enumitem}
  \usepackage{graphicx}

  \usepackage[center,loose,nooneline]{subfigure}
  
\makeatletter \theoremstyle{plain}
\newtheorem{thm}{Theorem}[section]
\newtheorem{lem}[thm]{Lemma}

\newtheorem{prop}[thm]{Proposition}

\theoremstyle{remark}
\newtheorem{rmk}[thm]{Remark}

\theoremstyle{definition}
\newtheorem{defin}[thm]{Definition}

 \newcommand{\N}{\mathbb{N}} 
 
 \newcommand{\R}{\mathbb{R}}

 \newcommand{\sphere}{\mathbb{S}} 
 
 \newcommand*\colvec[3][]{\begin{pmatrix}\ifx\relax#1\relax\else#1\\\fi#2\\#3\end{pmatrix}}

 \newcommand{\Hhh}{\mathscr{H}} 
 \newcommand{\Lll}{\mathscr{L}}

 \newcommand{\dist}{\mathrm{dist}}

 \newcommand{\dm}{\mathrm{d}}

  \newcommand{\0}{\mathbf{0}}

  \newcommand{\woz}{\backslash\{0\}}

\newcommand\ndot{{\mkern 2mu\cdot\mkern 2mu}}

 \newcommand{\lm}{\varnothing}  

\begin{document}

\title{Marstrand type projection theorems for normed spaces}
\author{Zolt\'an M. Balogh and  Annina Iseli}

\address {Mathematisches Institut,
Universit\"at Bern,
Sidlerstrasse 5,
CH-3012 Bern,
Switzerland}

\email{zoltan.balogh@math.unibe.ch}
\email{annina.iseli@math.unibe.ch}


\thanks{ This research was supported by the Swiss National Science Foundation Grant Nr. 200020 165507 .}
\begin{abstract}
We consider Marstrand type projection theorems for closest-point projections in the normed space $\R^2$.
We prove that if a norm on $\R^2$ is regular enough, then the analogues of the well-known statements from the Euclidean setting hold, while they fail for norms whose unit balls have corners. We establish our results by verifying Peres and Schlag's transversality property and thereby also obtain a Besicovitch-Federer type characterization of purely unrectifiable sets. \end{abstract}
\maketitle

\section{Introduction}
Marstrand's projection theorem~\cite{Marstrand1954} states that for a Borel set $A\subset \R^2$ of Hausdorff dimension $\dim A\geq 1$, for almost every line $L$ in $\R^2$ that passes through the origin, the image of $A$ under the orthogonal projection onto the line $L$ is of Hausdorff dimension~$1$. This result marked the start of a large series of results on the distortion of different notions of dimension under various classes of mappings, to which many authors have contributed over the last decades. In particular,
Kaufman~\cite{Kaufman1968}, Mattila~\cite{Mattila1975} and Falconer~\cite{Falconer1982} refined Marstrand's estimate for the size of the exceptional set of lines, added the case $\dim A <1$ and generalized  the results to higher dimensions.
Similar problems have been studied in various non-Euclidean settings: Balogh et.\,al \cite{BFMT2012},\cite{BDCFMT2013}, as well as Hovila~\cite{Hovila2014} have established Marstrand type projection theorems for the family of isotropic projections on the Heisenberg groups. Similar results have been proven for the family of orthogonal projections along geodesics on simply connected Riemannian surfaces of constant sectional curvature, by the authors of this paper~\cite{BaloghIseli2016}. While \cite{BFMT2012} and \cite{BDCFMT2013} adapted Kaufman's respectively Mattila's methods in order to obtain their results, \cite{Hovila2014} as well as \cite{BaloghIseli2016} employed the general projection theory due to Peres and Schlag~\cite{PS2000}. 
For a more extensive account on classical and recent projection theory, we recommend the expository articles \cite{Mattila2004} and \cite{FFJ2015}, as well as the textbooks \cite{Mattila1995} and \cite{Mattila2015}. \\[4pt]
In the general metric setting, the notion of orthogonal projections is not defined. Therefore it is more natural to replace it by the notion of closest-point projections. In particular, in Euclidean space orthogonal projections coincide with closest-point projections.\\
First, notice that for any strictly convex norm $N$ on $\R^2$ (that is, the closed unit ball $B_N(0,1)$ with respect to $N$ is a strictly convex set) the closest-point projection onto any straight line $L$ is well-defined. These closest-point projections with respect to $N$ are Lipschitz mappings and therefore do not increase Hausdorff dimension. Therefore, the question for a generic lower bound on the dimension of the projected set is a natural~one.\\
Since any two norms on $\R^2$ are bi-Lipschitz equivalent and the Hausdorff dimension is invariant under bi-Lipschitz maps, one might think that in the normed spaces setting, Marstrand type projection theorems are an easy consequence of their well-known Euclidean analogues. However, this is not case: As we shall see in Sections~\ref{sec_examples} and \ref{sec_final}, Marstrand type theorem fail for norms whose unit sphere has corner type singularities.
On the other hand, we prove that Marstrand type results do hold whenever $N$ is regular enough. To formulate our results, for any direction $v\in \sphere^{1}$ let us denote the line through the origin perpendicular to $v$ by $H_v$, and let $P_v:\R^2\to H_v$ be the closest point projection onto $H_v$ with respect to the norm $N$. For $s>0$, we denote by $\Hhh^s$ the Hausdorff $s$-measure on (subspaces of) $\R^2$ and by $\dim(A)$ the Hausdorff dimension of a set $A\in\R^2$. 

\begin{thm}\label{thm1} Let $(\R^2,N)$ be a normed space, where the restriction of the norm $N$ to $\R^2\woz$ is of class $C^{2,\delta}$ for some $\delta>0$ and the unit circle with respect to $N$ has strictly positive curvature. Let $A\in \R^2$ be a Borel set and let $s:=\dim A$. Then, the following hold:
\begin{enumerate}[itemsep=3pt]
\item if $s>1$ then
 \begin{enumerate}[label=(1.\alph*), itemsep=2pt]
\item $\Hhh^{1}(P_v(A))>0$ for $\Hhh^{1}$-a.e. $v\in\sphere^{1}$.
\item $\dim\{v\in\sphere^{1}: \Hhh^{1}(P_v(A))=0\}\leq 2-\min\{s,1+\delta\}$.
\end{enumerate}
\item if $s\leq 1$, then \begin{enumerate}[label=(2.\alph*), itemsep=2pt]
\item $\dim(P_v(A))=\dim(A)$ for $\Hhh^{1}$-a.e. $v\in\sphere^{1}$,
\item $\dim\{v\in\sphere^{1}: \dim(P_v(A))<s\}\leq s$.
\end{enumerate}
\end{enumerate}
\end{thm}

The proof of Theorem~\ref{thm1} consists of two parts: First, we will give an explicit formula~\eqref{proj_formula} for the projections $P_v:\R^2\to H_v$, $v\in\sphere^1$, in terms of the Gauss map. (Recall that the Gauss map by definition assigns to each point on the $N$-unit sphere its unit outward normal.) Having this explicit formula at hand we can apply the general projection theorem of Peres and Schlag to prove Theorem~\ref{thm1}. To check the assumptions of their theorem (that are sufficient regularity as well as the transversality property) is a non-trivial matter. This is done in Theorem~\ref{mainlemma}, which is our main technical result.
As we shall see, Theorem~\ref{mainlemma} implies Theorem~\ref{thm1} immediately (Section~\ref{sec_proofs_2dim}). \\[4pt]
As follows from a recent result of Hovila et.\,al~\cite{HJJL2012}, another remarkable consequence of Theorem~\ref{mainlemma} is the following Besicovitch-Federer \cite{Besic1939}, \cite{Federer1947} type characterization of purely unrectifiability.
\begin{thm}\label{thm2}
Assume that the assumptions of Theorem~\ref{thm1} hold and let $A$ be a $\Hhh^1$-measurable subset of $\R^2$ with $\Hhh^1(A)<\infty$. Then, $A$ is purely $1$-unrectifiable if and only if $\Hhh^1(P_v(A))=0$ for $\Lll^1$-a.e. $v\in\sphere^1$.
\end{thm}
Our paper is organized as follows: Section~\ref{sec_prelim} is for preliminaries, including a brief introduction to   the general projection theory due to Peres and Schlag~\cite{PS2000}. In Section~\ref{sec_proj}, we give a precise definition of our setting and derive an explicit formula for the family of closest-point projections in terms of the Gauss map. In Section~\ref{sec_proofs_2dim}, we formulate Theorem~\ref{mainlemma} and prove our main results. In Section~\ref{sec_examples}, we study the example of $p$-norms on $\R^2$ and thereby prove the necessity of regularity assumptions on $N$. Finally, in Section~\ref{sec_final}, we will give local versions of our main results and briefly address the higher dimensional setting.\\[6pt]
\textit{Acknowledgements:} The authors thank Katrin F\"assler, Pertti Mattila and Tuomas Orponen for valuable discussions on the subject of projection theorems. We also thank the referee for a careful reading of the paper and for helpful remarks.

\section{Preliminaries}\label{sec_prelim}

\subsection{H\"older spaces $C^{k,\delta}$}\label{sec_hoelder}

Let $U$ be an open subset of $\R^n$ and $0<\delta\leq 1$ and $k\in \N_0$. We say that $f:U\to\R^m$ is $\delta$\textit{-H\"older} if there exists $C_{\delta}$ such that for all $x,y\in U$, the estimate $\| f(x)-f(y)\|\leq C_\delta \|x-y\|^\delta$ holds. Furthermore, we say that $f:U\to\R^m$ is 
of class $C^{k,\delta}$ if $f$ is $k$-times continuously differentiable (i.e. $f$ is of class $C^k$) and its (partial) derivative(s) of order $k$ are locally $\delta$-H\"older, that is, for any $K\subset U$ compact, there exists a constant $M_K$ such that for all $x,y\in K$ and $\alpha$ multi-index with $|\alpha |=k$,
$$\|D^\alpha f (x)- D^\alpha f (y) \|\leq M_K \|x-y\|^\delta.$$

Note that, in particular, $C^{k,\delta}(U)\subset C^{k}(U)$ for all $0<\delta\leq 1$.\\

The class of $C^{k,\delta}$ mappings has many properties in common with the class of $C^k$ mappings. In particular, it is easy to check, that products and quotients with non-vanishing denominator of mappings of class $C^{k,\delta}$ are themselves $C^{k,\delta}$. Also, whenever $f,g$ are of class $C^{k,\delta}$ for some $k\in \N$, $0<\delta<1$, then $f\circ g$ is of class $C^{k,\delta^2}$. The following theorem is a version of the inverse function theorem for H\"older spaces:

\begin{thm}\label{thm_InverseFunctionTheorem}
Let $f:U\to \R^n$ be a mapping of class $C^{k,\delta}$ for an open set $U\subset \R^n$ that contains~$0$, where $k\in \N$ and $0<\delta<1$. Assume that $Df(0):\R^n\to \R^n$ is a linear diffeomorphism, then $f$ has a local inverse $f^{-1}$ at $0$ and $f^{-1}$ is of class $C^{k,\delta}$.
\end{thm}
The proof for Theorem~\ref{thm_InverseFunctionTheorem} is well-known and can be found in \cite{Hajlasz2005}.
 
\subsection{The projection theorem of Peres and Schlag} \label{sec_PS}
In this section, we briefly introduce a specific version of Peres-Schlag's projection theorem \cite{PS2000}. For a more general presentation we also suggest \cite{Mattila2015} and~\cite{Mattila2004}.\\[3pt]
Let $(\Omega,\dm)$ be a compact metric space, $J\subset \R$ an open interval and $\Pi$ a continuous map
\begin{equation}\label{eq_Pi}
\Pi: J\times \Omega \rightarrow \R, \ \
(\lambda, \omega) \mapsto \Pi(\lambda,\omega).
\end{equation}
We think of $\Pi$ as a family of projections $\Pi_\lambda \omega:=\Pi(\lambda,\omega)$ over the parameter interval $J$. Let $\lambda\in J$ and $\omega_1,\omega_2\in \Omega$ two distinct points. We define
\begin{equation}\label{def_Phi}
\Phi_\lambda(\omega_1, \omega_2)=\frac{\Pi_\lambda \omega_1 - \Pi_\lambda \omega_2}{\dm(\omega_1,\omega_2)}.
\end{equation}

All crucial properties for the abstract projection theorem  are collected in the following definition:
\begin{defin}\label{def_PS}
 \begin{enumerate}[label={(\alph*)},topsep=0pt,itemsep=0pt,partopsep=0pt,parsep=0pt] \item We say that $\Pi$ is $C^{1,\delta}$-regular for some $0<\delta<1$ if for any compact interval $I\subset J$, there exists constants $C_{I,\delta}>0$ such that:
\begin{itemize}[label=\raisebox{0.25ex}{\tiny$\bullet$}, itemsep=3pt]
\item  for all $\lambda\in I$ and $\omega \in \Omega$, $\left|\frac{\dm}{\dm\lambda}\Pi(\lambda,\omega) \right|\leq C_{I,\delta},$ i.e., $\frac{\dm}{\dm\lambda}\Pi:I\times \Omega \to \R$  is bounded,
\item for all $\lambda_1,\lambda_2\in I$ and $\omega \in \Omega$:
$$\left|\frac{\dm}{\dm \lambda} \Pi(\lambda_1,\omega) - \frac{\dm}{\dm \lambda} \Pi(\lambda_2, \omega)\right|\leq C_{I,\delta} \left| \lambda_1-\lambda_2\right|^\delta.$$
i.e. $\lambda \mapsto  \frac{\dm}{\dm\lambda}\Pi(\lambda, \omega)$ (for fixed $\omega \in \Omega$) is $\delta$-H\"older on $I$.
\end{itemize}
\item We say that $J$ is an interval of transversality for $\Pi$ (or shorter: $\Pi:J\times \Omega\to \R$ satisfies transversality), if there exists a constant $C>0$, such that for all pairs of distinct points $\omega_1,\omega_2\in \Omega$ and $\lambda\in J$,
for which $|\Phi_\lambda(\omega_1,\omega_2)|\leq C $,
 \begin{equation*}
 \left|\frac{\dm}{\dm \lambda}\Phi_\lambda(\omega_1,\omega_2)\right|\geq C.
\end{equation*}
\item We say that $\Phi$ is $(1,\delta)$-regular for some $0<\delta<1$ if there exist $C_\delta>0$ and $\tilde{C}>0$ such that, whenever $|\Phi_{\lambda_1}(\omega_1,\omega_2)|+ |\Phi_{\lambda_2}(\omega_1,\omega_2)|\leq C $ (where $C$ as in (b)) for $\omega_1\neq \omega_2\in \Omega$ and $\lambda_1,\lambda_2\in J$, then,
\begin{itemize}[label=\raisebox{0.25ex}{\tiny$\bullet$}, itemsep=3pt]
\item $ \left |\frac{\dm}{\dm \lambda} \Phi_{\lambda_1} (\omega_1,\omega_2)\right| \leq \tilde{C}$,
\item $  \left|\frac{\dm}{\dm \lambda} \Phi_{\lambda_1} (\omega_1,\omega_2) - \frac{\dm}{\dm \lambda} \Phi_{\lambda_2} (\omega_1,\omega_2)\right|\leq C_\delta \left| \lambda_1-\lambda_2\right|^\delta$.
\end{itemize}
\end{enumerate}
\end{defin}

The proofs of Theorem~\ref{thm1} and Theorem~\ref{thm2} will be based on the following abstract projection theorem. It is a special case of of Theorem 4.9 in \cite{PS2000}.

\begin{thm}\label{thmPS}
Let $\Omega$ be a compact metric space which is bi-Lipschitz equivalent to a subset of Euclidean space; $J$ an open interval and $\Pi$ a continuous map as described in \eqref{eq_Pi}. Assume that conditions (a), (b) and (c) of Definition~\ref{def_PS} are satisfied (for some $0<\delta<1$).
Then the following statements hold for all Borel sets $A\subseteq \Omega$. \begin{enumerate}[itemsep=3pt]
\item If $\dim A > 1$, then 
\begin{enumerate}[itemsep=2pt]
\item  $\Lll^1 (\Pi_\lambda A)>0$ for $\mathscr{L}^1$-a.e. $\lambda\in  J$,
\item $ \dim\{\lambda \in  J\ : \ \mathscr{L}^1(\Pi_\lambda A)=0\}\leq 2-\min\{\dim A,1+\delta\}$.
\end{enumerate} 
\item If $\dim A \leq 1$, then 
\begin{enumerate}[itemsep=2pt]
\item  $\dim (\Pi_\lambda A)=\dim A$ for $\mathscr{L}^1$-a.e. $\lambda\in  J$,
\item For $0<\alpha\leq\dim A$, $\dim\{\lambda \in  J\ : \ \dim(\Pi_\lambda A)<\alpha\}\leq \alpha$.
\end{enumerate}
\end{enumerate}
\end{thm}

\section{Projection formula and Gauss map}\label{sec_proj}

We start by properly introducing our setting and give a formal definition of the Gauss map. Also, we will express the assumptions from Theorem~\ref{thm1} in terms of the Gauss map. In the second part of this section, we derive an explicit formula for the projections $P_v:\R^2\to H_v$.

\subsection{Normed spaces, convexity and Gauss map}

Consider the normed space $(\R^2,N)$ for an arbitrary norm $N$ and by $\|\ndot\|$ denote the Euclidean norm on $\R^2$. Denote closed balls with center $x_0\in\R^2$ and radius $r>0$ with respect to $N$ by  
$B_N(x_0,r):= \{x\in\R^2: N(x-x_0)\leq r\},$
and with respect to $\|\ndot\|$ by
$B_E(x_0,r):= \{x\in\R^2: \|x-x_0\|\leq r\}.$ 
We will write $\partial A$ for the boundary of a set $A\subset \R^2$ and in most cases denote $\partial B_E(0,1)$ by $\sphere^{1}$. 
Furthermore, denote the distance of sets $A,B\subseteq \R^2$ with respect to $N$ by $$\dist_N(A,B):=\inf\{ N(a-b):a\in A, \ b\in B \}$$ and analogously define $\dist_E$ depending on $\|\ndot\|$.\\

There exist many equivalent definitions for (strictly) convex subsets of Euclidean space. We will use the following one: 

\begin{defin}
A closed set $A\subset \R^2$ with non-empty interior is called \textit{convex} if for every point $x\in \partial A$, there exists a straight line $L$ through $x$ such that $A$ lies in the closed half space to one side of $L$. Moreover, in case $\partial A$ is a $C^1$-manifold, this line is unique and $L=x+T_x \partial A$. The set $A$ is called \textit{strictly convex} if in addition $L\cap A=\{x\}$.
\end{defin} 

By the definition of a normed space, the following Proposition trivially holds.

\begin{prop}\label{prop0} Let $N$ be a norm on $\R^2$, then:
\begin{enumerate} 
\item $B_N(x_0,r)=x_0+r\ndot B_N(0,1):=\{x_0+rx:x\in B_N(0,1)\}$ for all $x_0\in\R^2$ and $r>0$,
\item $B_N(0,1)$ is a convex set that contains the origin and that is mapped to itself by the antipodal map $x\mapsto -x$.
\item There exists a constant $L>0$, such that $$\tfrac{1}{L}\, B_E(0,1)\subset B_N(0,1) \subset L\, B_E(0,1).$$
\item $B_N(0,1)$ determines $N$. In other words: For any convex set $B$ with non-empty interior that contains the origin and is mapped to itself under the antipodal map, there exists a unique norm $N$ such that $B=B_N(0,1)$.
\end{enumerate}\end{prop}

This allows us to define the Gauss map for $\partial B_N(0,1)$:

\begin{prop}\label{prop1}
Let $N$ be a norm on $\R^2$ such that $N$ restricted to $\R^2\woz$ is of class $C^k$, $k\in\N$. Then $\partial B_N(0,1)$ is an $1$-dimensional $C^k$-manifold in $\R^2$. In particular, this means that the Gauss map $G:\partial B_N(0,1)\to \sphere^{1}$, which assigns the outward normal vector to each point on $\partial B_N(0,1)$, is well-defined and is given by $$G(x)=\frac{\nabla N(x)}{\|\nabla N(x)\|},$$ where $G$ is of class $C^{k-1}$.
\end{prop}

The following propositions summarize the properties of the Gauss map that we shall need in the sequel.

\begin{prop}\label{prop2}
Let $N$ be a norm on $\R^2$, such that $N$ restricted to $\R^2\woz$ is of class $C^1$, then:
\begin{enumerate}
\item $x \ndot G(x)\neq 0$ for all $x\in \partial B_N(0,1)$, where $\ndot$ denotes the scalar product in $\R^2$.
\item $G$ is surjective. 
\item $N$ is strictly convex if and only if $G$ is injective.
\item If $N$ restricted to $\R^2\woz$ is of class $C^2$, then $G$ is of class $C^1$.
\end{enumerate}
\end{prop}

Statement (1) is proven as follows: Assume for a contradiction that there exists $x\in \partial B_N(0,1)$ such that $x\ndot G(x)=0$. By definition of $G$, this means that $x\in T_x \partial B_N(0,1)$, and since $B_N(0,1)$ is convex, $B_N(0,1)$ lies on one side of the line $H=T_x \partial B_N(0,1)$. Thus, by antipodal symmetry of $B_N(0,1)$ and the fact that $0\in T_x \partial B_N(0,1)$, it follows that $B_N(0,1)\subset T_x \partial B_N(0,1)$. However, this contradicts the fact that $B_N(0,1)$ has non-empty interior. And thus (1) follows.
Statements (2) to (4) from Proposition~\ref{prop2}
are standard.\\[6pt]
Let $N$ be a norm on $\R^2$ such that $N$ restricted to $\R^2\woz$ is of class $C^2$. Then, by Proposition~\ref{prop0} and Proposition~\ref{prop1}, $\partial B_N(0,1)$ is a simply closed $C^2$-curve. Let $\gamma:\R\to \partial B_N(0,1)\subset \R^2$ be the counter-clockwise $C^2$-parametrization of $\partial B_N(0,1)$ by arc-length such that $\gamma(t)=\gamma(t+l)$ for all $t\in \R$ and $l$ is the Euclidean length of $\partial B_N(0,1)$. Then the curvature of $\partial B_N(0,1)$ at $x=\gamma(t)\in \partial B_N(0,1)$ is $K(\gamma(t))=\|\ddot{\gamma}(t)\|$ and the Gauss map is given by
\begin{equation}\label{G1}
G(\gamma(t))=\frac{\ddot{\gamma}(t)}{\|\ddot{\gamma}(t)\|}= \frac{\ddot{\gamma}(t)}{K(\gamma(t))}.
\end{equation}
 Moreover, we know that, since $\gamma$ is parametrized by arc-length, $\|\dot{\gamma}(t)\|=1$ and $\dot{\gamma}(t)$ is perpendicular to $\ddot{\gamma}(t)$ at any $t\in \R$. Thus we can write $G$ as
\begin{equation}\label{G2}
G(\gamma(t))=R(\dot{\gamma}(t)),
\end{equation} 
for all $t\in\R$, where $R$ is the counter-clockwise rotation about an angle $\tfrac{\pi}{2}$ in $\R^2$. 

\begin{prop}\label{prop3}
Let $N$ be a norm on $\R^2$, such that $N$ restricted to $\R^2\woz$ is of class $C^2$. Then $G:\partial B_N(0,1)\to \sphere^1$ is a $C^1$-diffeomorphism if and only if $K\neq 0$. Moreover, in this case, $B_N(0,1)$ is strictly convex.
\end{prop}

\begin{proof}[Proof of Proposition~\ref{prop3}]Assume that $K>0$ everywhere on $\partial B_N(0,1)$. 
First we show that $G$ is a local $C^1$-diffeomorphism: By \eqref{G2} and the linearity of the rotation $R$,  $\frac{\dm}{\dm t}G(\gamma(t))=R\circ \ddot{\gamma}(t)$. Thus, since $\ddot{\gamma}(t)$ is normal to $\sphere^1$ at $G(\gamma(t))$,  $\frac{\dm}{\dm t}G(\gamma(t))$ is tangent to $\sphere^1$ at $G(\gamma(t))$. (Recall that $\ddot{\gamma}(t)\neq 0$ by the assumption $K>0$.) Thus the map $\frac{\dm}{\dm t}G(\gamma(t)): T_{\gamma(t)}\partial B_N(0,1) \to T_{G(\gamma(t)} \sphere^1$ is surjective and so, by inverse function theorem, $G$ is a local $C^1$-diffeomorphism.
Now, we show that $B_N(0,1)$ is strictly convex and thus $G$ is a $C^1$-diffeomorphism. To this end, let $K>0$ and assume that $B_N(0,1)$ is not strictly convex, then, since it is convex, its boundary $\partial B_N(0,1)$ has to contain a line segment. However, this would imply, that the curvature on this part of the boundary must be zero, which contradicts our assumption. Therefore, $ B_N(0,1)$ must be strictly convex and hence, by Proposition~\ref{prop2}, $G$ is a $C^1$-diffeomorphism.
For the converse, assume that $G$ is a $C^1$-diffeomorphism and recall that $G(\gamma(t))=R\dot{\gamma}(t)$. Therefore, $\frac{\dm}{\dm t}G(\gamma(t))=R\ddot{\gamma}(t)$ is in $T_{G(\gamma(t))}\sphere^1\woz$. Hence, $\ddot{\gamma}(t)$ is non-zero for all $t\in \R$ and thus the curvature is $K(\gamma(t))=\|\ddot{\gamma}(t)\|>0$ for all $t\in\R$.
\end{proof}

\subsection{Projection formula}

Let $N$ be a strictly convex norm on $\R^2$. For any direction $v\in\sphere^{1}$, by $H_v$ we denote the the line that is perpendicular to $v$ and contains the origin: 
$$H_v:=\{w\in \R^2: v\ndot w=0\}.$$ 

Define the (closest point) projection $P_v:\R^2\to H_v$ onto $H_v$ as follows: For $x\in\R^2$ let $P_v(x)\in H_v$ such that 
$
N(x-P_v(x))=\min\{N(x-q):q\in H_v\}.$ Thus the closest point projection can be viewed as a map on the product space $\sphere^{1}\times \R^2$:  \begin{equation*}\begin{split}
P:\sphere^{1}\times \R^2 &\ \to \  \R^2\\ 
(v,x)& \ \mapsto \ P(v,x):=P_v(x)
\end{split}
\end{equation*}

Since all lines through the origin in $\R^2$ are convex sets and $N$ is strictly convex, the closest-point projection is well defined. 
Assume that $N$ is such that the Gauss map $G:\partial B_N(0,1)\to \sphere^{1}$ is well-defined and bijective. Our goal is to derive an explicit formula for $P_v$ in terms of the Gauss map $G$. For this, let $v\in \sphere^{1}$ and $x\in \R^2\backslash H_v$ arbitrary. 

\begin{rmk}\label{rmk_proj_intuitively}
Intuitively, the easiest way to think of the projection $P_v(x)$ is the following: Consider $B(x,\epsilon)$ with $0<\epsilon<\dist_N(p,H_v)$.
Start enlarging $\epsilon$ until $\partial(B_N(x,\epsilon))$ meets $H_v$. 
This will be the case when $\epsilon$ equals $\dist_N(x,H_v)$, then $P_v(x)$ will be the intersection point of $\partial(B_N(x,\epsilon))$ and $H_v$, see Figure~\ref{fig_intuition}. 
\begin{figure}[h]
\begin{center}
\def\svgwidth{250pt}
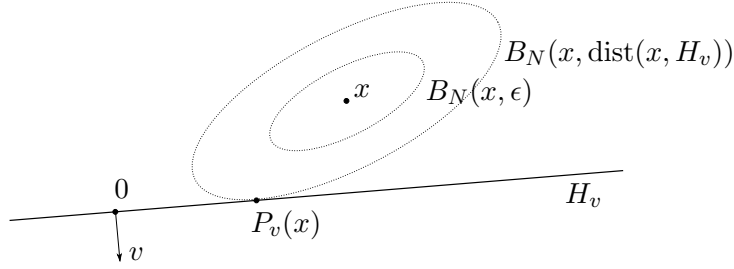
\end{center}
\caption{Intuition for the closest-point projection onto $H_v$.}\label{fig_intuition}
\end{figure}
\end{rmk}

Formally, the construction of $P$ goes as follows:
Recall that we denote by $G$ the Gauss map that assigns to each $x\in\partial B_N(0,1)$ its outward normal $G(x)\in\sphere^{1}$ and that we can express $G$ in terms of the norm $N$ (see Proposition~\ref{prop1}). Define $\bar{G}:\R^2\backslash \{0\} \to \sphere^{1}$ by $$\bar{G}(x):=G\left(\frac{x}{N(x)}\right).$$ Thus $\bar{G}$ is an extension of $G$ and therefore surjective. 
Furthermore, by homogeneity of the norm, $\bar{G}$ assigns to $x \in \R^2\backslash \{0\}$ the outward normal of $\partial B_N(0,N(x))$, and, since $G$ is bijective, $\bar{G}$ restricted to $\partial B_N(0,r)$ is bijective for all $r>0$.\\
Again let $v\in \sphere^{1}$ and $x\in \R^2\backslash H_v$, then $P_v x$ is characterized by the two following facts: First, $P_vx$ lies on the line $H_v$, this means that $P_vx$ it is perpendicular to $v$, i.e., $(P_vx)\ndot v=0$. Second, $P_v(x)\in B_N(0,\dist(x,H_v))$ and the outward normal of $B_N(x,\dist(x,H_v))$ at $P_v(x)$ is perpendicular to $H_v$. This means that, either $\bar{G}(P_v(x)-x)=v$ or $\bar{G}(P_v(x)-x)=-v$, depending on which side of $H_v$ the point $x$ lies. First, we consider the case when $x$ and $v$ lie on opposite sides of $H_v$ , i.e., when $\bar{G}(P_v(x)-x)=v$, see Figure~\ref{fig1}.
Note that, since $P_v$ is well defined, the system \eqref{eq1} has a unique solution.
Let $\tilde{q}$ be the point where the ray from $x$ in direction $q$ and the sphere $\partial B_N(x,1)$ intersect.  Then 
\begin{equation}\label{eq_for_q}
(q-x)=\lambda (\tilde{q}-x).
\end{equation}  

\begin{figure}[h]
\begin{center}
\def\svgwidth{250pt}
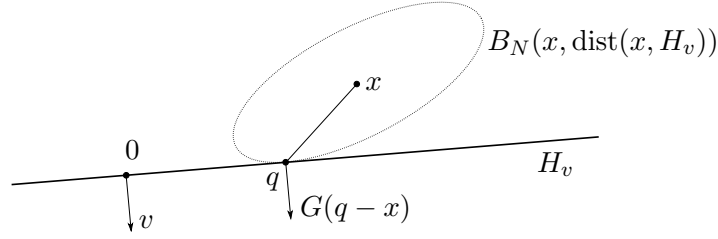
\end{center}
\caption{Closest-point projection onto $H_v$ (when $x$ and $v$ lie on the opposite side of $H_v$).}\label{fig1}
\end{figure}
Thus $P_v(x)$ is the solution $q$ of the following system of equations:
\begin{equation}\label{eq1}
\begin{cases} \bar{G}(q-x)= v \\
v\ndot q=0 \end{cases}
\end{equation}

for $\lambda=N(q-x)>0$ and $\bar{G}(q-x)=G(\tilde{q}-x)$. Thus $\bar{G}(q-x)=v$ is equivalent to $G(\tilde{q}-x)=v$ and therefore also to $\tilde{q}-x=G^{-1}(v)$. Furthermore, by \eqref{eq_for_q}: $q=x+(q-x)=x+\lambda(\tilde{q}-x)$. Hence, \eqref{eq1} is equivalent to
\begin{equation}\label{eq2}
\begin{cases} 
\tilde{q}-x=G^{-1}(v) \\
v\ndot (x+\lambda(\tilde{q}-x))=0 
\end{cases}
\end{equation} 
and since \eqref{eq1} has a unique solution, so does \eqref{eq2}. 
\begin{rmk}\label{rmk_x_times_G(x)_non-zero}
From (1) in Proposition~\ref{prop2}, compactness of $\sphere^{1}$ and continuity of $G$, it follows that $v\ndot G^{-1}(v)$ is larger than some positive constant depending on $N$. \end{rmk} 

From Remark~\ref{rmk_x_times_G(x)_non-zero} and \eqref{eq2} it follows that 
$$\lambda=-\frac{v\ndot x}{v\ndot(x-\tilde{q})}=-\frac{v\ndot x}{v\ndot G^{-1}(v)},$$
by \eqref{eq_for_q}, $$q=x+(\tilde{q}-x)\lambda=x-G^{-1}(v)\frac{v\ndot x}{v\ndot(x-\tilde{q})},$$
and therefore,
\begin{equation*}
P_v(x)=q=x-G^{-1}(v)\frac{v\ndot x}{v\ndot G^{-1}(v)}
\end{equation*}

Now consider the case when $x$ and $v$ lie on the same side of $H_v$, that is $\bar{G}(P_v(x)-x)=-v$. Then, by repeating the above argument, we obtain 
\begin{equation*}
P_v(x)=q=x-G^{-1}(-v)\frac{(-v)\ndot x}{(-v)\ndot G^{-1}(-v)}=x-G^{-1}(v)\frac{v\ndot x}{v\ndot G^{-1}(v)}.
\end{equation*}
Thus, in both cases, ($\bar{G}(P_v(x)-x)=v$ and $\bar{G}(P_v(x)-x)=-v$), we obtain the same formula and therefore conclude that:

For any $x\in\R^2$ and $v\in \sphere^{1}$, the closest point projection of $x$ onto the line $H_v$ orthogonal to $v$ is given by 
\begin{equation}\label{proj_formula}
P_v(x)=P(v,x)=x-G^{-1}(v)\frac{v\ndot x}{v\ndot G^{-1}(v)}
\end{equation}

Note that \eqref{proj_formula} in particular shows that, for fixed $v$, $P_v:\R^2\to H_v$ is a linear map. Linear maps of finite dimensional spaces are Lipschitz and hence do not increase the Hausdorff dimension of sets. Therefore, for measurable sets $A\in \R^2$, we have the following trivial estimate:
$$\dim P_vA\leq \min\{\dim A, 1 \}.$$
Also, the following lemma is an immediate consequence of \eqref{proj_formula}.

\begin{lem}\label{lem_v0}
For a point $x\in \R^2\woz$ and a direction $v\in\sphere^{1}$: $P_v(x)=0$ if and only if $x$ is co-linear to $G^{-1}(v)$. In particular, by bijectivity of $G:\partial B_N(0,1)\to \sphere^{1}$ it follows that for every point $x$ in $\R^2\woz$, there exist a vectors $v_0\in\sphere^{1}$ for which $P_{v_0}(x)=0$, and there exists only only vector $v\in \S^1\backslash\{v_0\}$ for which $P_v(x)=0$ as well, which is $v=-v_0$.
\end{lem}

Note that by linearity of $x\mapsto P_v(x)$, the statement $P_v(x)=P_v(y)$ is equivalent to $P_v(x-y)=0$, for $x,y\in \R^2$, $x\neq y$. Therefore, by Lemma~\eqref{lem_v0}, for  given  $x,y\in \R^2$, $x\neq y$, there exists a (up to the choice between a vector $v$ and its antipodal vector $-v$) unique $v_0=v_0(x-y)$ such that $P_{v_0}(x)=P_{v_0}(y)$.

\section{Proof of the main results}\label{sec_proofs_2dim}

We begin by adjusting the terminology so that Theorem~\ref{thmPS} can be applied to our setting:
First we need to identify $\sphere^1$ with (an interval in) the real numbers. To this end, consider the parametrization $v:\R\to \sphere^1\subset \R^2$, given by $v(t)=(\cos t ,\sin t)$. Thus, $v$ parametrizes $\sphere^1$ by arc-length in counter-clockwise direction. In particular, for every $t\in \R$, $(v(t),\dot{v}(t))$ is a positively oriented orthonormal basis of $\R^2$.
Recall that we denote the counter-clockwise rotation in $\R^2$ by an angle $\tfrac{\pi}{2}$ by $R$. Thus for all $v\in \sphere^1$, $(v,Rv)$ is a positively oriented orthonormal basis of $\R^2$, and for $v(t)$ with $t\in\R$, $\dot{v}(t)=Rv(t)$. Recall that for all $v\in \sphere^1$, we denote the line in $\R^2$ through $0$ perpendicular to $v$ by $H_v$. Thus, in particular, $H_v$ can be written as $H_v=\{tRv:t\in \R\}$. Furthermore, recall that $P(v,x)$ denotes the closest point projection of a point $x\in \R^2$ onto $H_v$ with respect to $N$, and that equation \eqref{proj_formula} is an explicit formula for $P(v,x)$. \\[3pt]
Recall that by $v\ndot w$ we denote the scalar product of two vectors $v$ and $w$ in Euclidean space and define $\Pi:\R\times \R^2 \to \R$ by 
\begin{equation}\label{def_Pi}
\Pi(t,x):= P(v(t),x)\cdot Rv(t),
\end{equation}
In the following remark, we list some properties of $\Pi$ that will be used in the sequel.

\begin{rmk}\label{rmk_properties_Pi} Assume that the restriction of the norm $N$ to $\R^2\woz$ is of class $C^{2,\delta}$ for some $\delta>0$ and that the Gauss map $G$ is a $C^1$-diffeomorphism. Consider $\Pi:\R\times\R^2\to\R$ $(t,x)\to \Pi(t,x)$, as defined in \eqref{def_Pi}. Then:
\begin{enumerate}[label=(\roman*), itemsep=3pt]
\item $\Pi$ is $2\pi$-periodic in $t$, that is: $\Pi(t,x)=\Pi(t+2\pi,x)$ for all $t\in \R$ and $x\in \R^2$.
\item For fixed $t\in \R$, $x\to \Pi(t,x)$ is linear and thus in particular a $C^\infty$ mapping.
\item Obviously, $t\mapsto \Pi(t,x)$ is a combination of products, quotients with non-vanishing denominator and composition of mappings of class $C^{1,\delta}$. (In particular, the fact that we chose $N$ to be of class $C^{2,\delta}$ outside of zero, implies that $G=\frac{\nabla N}{\|\nabla N\|}$ is of class $C^{1,\delta}$.)\\ Thus, for fixed $x\in \R^2$, $t\mapsto\Pi(t,x)$ is of class $C^{1,\delta}$
(see Section~\ref{sec_hoelder}).
\item As a conclusion of the two properties above: $\Pi$ is of class $C^{1,\delta}$ and in particular, $\Pi$ and $\frac{\dm}{\dm t}\Pi$ are continuous. 
\end{enumerate}
\end{rmk}

Let  $J\subset \R$ be an open interval and let $\Omega$ be a (Euclidean) ball of large radius, centered at the origin. The following is our main technical result.
\begin{thm}\label{mainlemma}
Under the assumptions of Theorem~\ref{thm1}, the family of projections $\Pi:J\times \Omega\to \R$ satisfies conditions (a), (b) and (c) from  Definition~\ref{def_PS}.
\end{thm}
We continue by proving Theorem~\ref{mainlemma} and at the end of this section deduce Theorem~\ref{thm1} and Theorem~\ref{thm2} from it.\\[6pt]
As we will shall prove now, the conditions from Definition~\ref{def_PS} can be simplified by an essential amount, using the linearity of the projections: By linearity of $x\mapsto P(v,x)$ for fixed $v\in \sphere^{1}$, it follows that the function $\Phi$ defined in \eqref{def_Phi} in our setting can be written as \begin{equation}\label{def_Phi_linear}
\Phi_t(x, y)=\frac{\Pi(t,x)-\Pi(t,y)}{N(x-y)}=\Pi\left(t,\frac{x-y}{N(x-y)}\right),
\end{equation}
for all $x\neq y\in \R^2$ and $t\in J$.
Note that for all $x\neq y\in \R^2$, $\frac{x-y}{N(x-y)}\in \partial B_N(0,1)$. Thus in order to check certain properties for $\Phi$, by \eqref{def_Phi_linear}, it suffices to very them for $\Pi$ restricted to $J\times \partial B_N(0,1)$. Also recall from Remark~\ref{rmk_properties_Pi} that whenever $N$ is a norm that satisfies the assumptions from Theorem~\ref{thm1}, then $\Pi:\R \times \R^2 \to \R$, given by \eqref{proj_formula} and \eqref{def_Pi}, is continuous and $\lambda\mapsto \Pi(\lambda,x)$ is differentiable for $\lambda\in \R$. Now we consider the restriction $\Pi:\bar{J} \times \Omega \to \R $.

\begin{lem}\label{lem_simplify_cond} Let $N$ be a norm that satisfies the assumptions from Theorem~\ref{thm1} and assume that $\Pi:\bar{J} \times \Omega \to \R$ defined in \eqref{proj_formula} and \eqref{def_Pi}  has the following additional properties: \begin{enumerate}[label=(P\arabic*)]
\item $\Pi:\bar{J}\times \Omega \to \R$ is bounded.
\item There exists $0<\delta<1$ and $C_{\delta}>0$ such that 
$\left|\frac{\dm}{\dm t} \Pi(t_1,x) - \frac{\dm}{\dm t} \Pi(t_2, x)\right|\leq C_{\delta} \left| t_1-t_2\right|^\delta,$ for all $x\in \Omega$, $t_1,t_2\in \bar{J}$.
\item Whenever $\Pi(t,x)=0$ for some $t\in \bar{J}$ and $x\in \partial B_N(0,1)$, then $\left|\frac{\dm}{\dm t}\Pi(t,x)\right|>0$.
\end{enumerate}
Then $\Pi:J\times \Omega\to \R$ satisfies the conditions (a), (b) and (c) from Definition~\ref{def_PS}. \end{lem}

\begin{proof} 
(P1) and (P2) immediately imply (a). Now, consider the restricted map $\Pi:J\times \partial B_N(0,1)\to \R$ and note that  by linearity \eqref{def_Phi_linear}: (P1) and (P2) imply (c). 

In order to prove (b), consider the restriction $\Pi:\bar{J}\times \partial B_N(0,1)\to \R$. Note that in order to establish (b), by linearity \eqref{def_Phi_linear}, it suffices to show that: There exists $C>0$ such that for $t\in \bar{J}$, $x\in \partial B_N(0,1)$ with $|\Pi(t,x)|<C$, it follows that $\left|\frac{\dm}{\dm t}\Pi(t,x)\right|>C$.
Assume, that this is not the case. Then, for every $n\in \N$ there exist $t_n\in \bar{J}$ and $x_n\in \partial B_N(0,1)$ such that $|\Pi(t_n,x_n)|<\tfrac{1}{n}$ and $\left|\frac{\dm}{\dm t}\Pi(t_n,x_n)\right|\leq\tfrac{1}{n}$. Since, $\bar{J}\times \partial B_N(0,1)$ is compact, the sequence $\{(t_n,x_n)\}_{n\in\N}$ admits a convergent subsequence with limit $(t,x)\in \bar{J}\times \partial B_N(0,1)$. Then by continuity of $\Pi$ and $\frac{\dm}{\dm t}\Pi$ (see Remark~\ref{rmk_properties_Pi}), $|\Pi(t,x)|=0$ and $\left|\frac{\dm}{\dm t}\Pi(t,x)\right|=0$, which contradicts (P3).
\end{proof}

Consider $P: \sphere^{1} \times \R^2\to \R^2, \ \ (v,x)\mapsto P(v,x),$ given by \eqref{proj_formula}.
The following lemma is the key tool for the prove of Theorem~\ref{mainlemma}. We will employ it in order to establish property (P3) for $\Pi:\bar{J}\times \R^2 \to \R$. 

\begin{lem}\label{lem_transv} Assume that $N$ restricted to $\R^2\woz$ is of class $C^2$ and $G$ is a $C^1$-diffeomorphism. 
Let $x\in \R^2\woz$ and $v_0\in \sphere^1$, be one of the two directions for which $P(v_0,x)=0$ and  $t_0\in \R$ such that $v_0=v(t_0)$. Then, the differential $\frac{\dm}{\dm t}P(v(t_0),x)\in\R^2$ is non-zero.
\end{lem}

\begin{proof}
Define the  map $\psi:\R\to \R$ by
$$\psi(t):=\frac{v(t) \ndot x}{v(t)\ndot G^{-1}(v(t))},$$
and fix $x\in\R^2\woz$, $t_0\in\R$ and $v_0=v(t_0)$ such that $P(v_0,x)=0$. The proof of Lemma~\ref{lem_transv} is based on the following fact:\\[3pt]
\underline{Claim:} $\psi(t_0)\neq 0$ and $\dot{\psi}(t_0)=0.$\\[3pt]
\underline{Proof of the Claim:} 
Since $x\neq 0$ and $P(v(t_0),x)=0$, \eqref{proj_formula} immediately implies that $\psi(t_0)\neq 0$. Now consider $\dot{\psi}(t)$ for $t\in \R$:
\begin{equation*}
\dot{\psi}(t)= \frac{ (\dot{v} (t)\ndot x)[v(t)\ndot G^{-1}(v(t))]
-(v(t)\ndot x)\left[\dot{v}(t)\ndot G^{-1}(v(t)) + v(t)\ndot DG^{-1}(v(t))(\dot{v}(t)) \right]  }{[v(t)\ndot G^{-1}(v(t))]^2}
\end{equation*}
Since $DG^{-1}(v(t))(\dot{v}(t))\in T_{G^{-1}(v(t))}\partial B_N(0,1)$ and $v(t)$ is orthogonal to $T_{G^{-1}(v(t))}\partial B_N(0,1)$, for all $t$, it follows that $v(t)\ndot DG^{-1}(v(t))(\dot{v}(t))=0,$  and hence 
\begin{equation*}
\dot{\psi}(t)= \frac{ (\dot{v} (t)\ndot x)\left[v(t)\ndot G^{-1}(v(t))\right]
-(v(t)\ndot x)\left[\dot{v}(t)\ndot G^{-1}(v(t))  \right]  }{[v(t)\ndot G^{-1}(v(t))]^2},
\end{equation*}
for all $t\in \R$.
Since $v(t)$ parametrizes $\sphere^1$ by arc-length (in counter-clockwise direction), $(v(t), \dot{v}(t))$ is a (positively oriented) orthonormal basis. This implies that
$$(\dot{v} (t)\ndot x)(v(t)\ndot G^{-1}(v(t))
-(v(t)\ndot x)(\dot{v}(t)\ndot G^{-1}(v(t))) = x\ndot R[G^{-1}(v(t))]$$
Recall that $R$ denotes the rotation in $\R^2$ by angle $\tfrac{\pi}{2}$. Now consider $t_0\in \R$ such that $v_0=v(t_0)$. Then, by Lemma~\ref{lem_v0},\, $  x\ndot R[G^{-1}(v(t_0))]=0$ and $$\dot{\psi}(t_0)=\frac{x\ndot R[G^{-1}(v(t_0))]}{ [v(t_0)\ndot G^{-1}(v(t_0))]^2 } = 0.$$
This proves the claim.\\[3pt]
Set $\psi(t):=\frac{v(t) \ndot x}{v(t)\ndot G^{-1}(v(t))}$. Now, we obtain,
\begin{equation*}
\begin{split}
\frac{\dm}{\dm t}P(v(t_0),x)
&= \ \frac{\dm}{\dm t}_{\big| _{t=t_0}}\left(  x-\psi(t) \,G^{-1}(v(t))\right)\\
&= \ \left(-\dot{\psi}(t)G^{-1}(v(t))-\psi(t)\, DG^{-1}(v(t))(v'(t))\right)\big|_{t=t_0}\\
& = \ -\dot{\psi}(t_0)G^{-1}(v(t_0))-\psi(t_0)\, DG^{-1}(v(t_0))(\dot{v}(t_0)),
\end{split}
\end{equation*} and thus, by the above claim,
\begin{equation}\label{eq_a1}
\frac{\dm}{\dm t}P(v(t_0),x)=-\psi(t_0)\, DG^{-1}(v(t_0))(\dot{v}(t_0)).
 \end{equation}
where $DG^{-1}(v)$ denotes the differential of the inverse Gauss map $G^{-1}$ at a point $v\in \sphere^1$.
Since $G^{-1}$ is a $C^1$-diffeomorphism and $\dot{v}(t)\neq 0$ for all $t\in \R$, $DG^{-1}(v(t_0))(\dot{v}(t_0))\neq 0$. Therefore, by the above claim, $\frac{\dm}{\dm t}P(v(t_0),x)\neq 0$
\end{proof}

\begin{proof}[Proof of Theorem~\ref{mainlemma}] By Proposition~\ref{prop3} and Lemma~\ref{lem_simplify_cond}, it suffices to establish properties (P1), (P2) and (P3) for the projection family $\Pi:\bar{J}\times \Omega \to \R$, where $N$ restricted to $\R^2\woz$ is of class $C^{2,\delta}$ and $G$ is a $C^1$-diffeomorphism. Recall from \eqref{proj_formula} and \eqref{def_Pi} that $$P_v(x)=x-G^{-1}(v)\frac{v\cdot x}{v\cdot G^{-1}(v)},$$ and $\Pi(t,x)=P(v(t),x)\cdot Rv(t),$ for $t\in \bar{J}$ and $x\in \Omega$. By Remark~\ref{rmk_x_times_G(x)_non-zero}, the compactness of $\sphere^1$ and $\Omega$, as well as the continuity of $G^{-1}$, there exists a constant $c>0$ such that $\|P_v(x)\|<c$ for all $x\in \Omega$ and $v\in \sphere^1$. Thus, since for all $v\in \sphere^1$ the mapping $w\mapsto w\cdot Rv$ restricted to $H_v$ is an isometry $H_v\to \R$, $\Pi: \bar{J}\times \Omega\to \R$ is bounded. This proves (P1).\\[3pt]
Fix $x\in \Omega$. 

Recall that by Remark~\ref{rmk_properties_Pi}, $t\mapsto \Pi(t,x)$ is of class $C^{1,\delta}$, that is, there exists a constant $C_{\delta,x}$ such that for all $t_1,t_2\in \bar{J}$
\begin{equation}\label{eq10}
\left|\frac{\dm}{\dm t} \Pi(t_1,x) - \frac{\dm}{\dm t} \Pi(t_2, x)\right|\leq C_{\delta,x,I} \left| t_1-t_2\right|^\delta
\end{equation}
Then, since $x\mapsto \Pi(t,x)$ is linear (for all $t\in J$), so is $x\mapsto\frac{\dm}{\dm t} \Pi(t, x)$ and hence,  $C_{\delta,I,x}$ in \eqref{eq10} can be replaced by a constant $C_{\delta}>0$ that is independent of $x\in \Omega$. This proves (P2).\\[3pt]
Let $x\in \R^2$ and $t_0\in \bar{J}$ such that $\Pi(t_0,x)=0$. Recall from \eqref{def_Pi}, that $\Pi(t,x)=P(v(t),x)\ndot Rv(t)$. The product rule for derivations yields that
\begin{equation}\label{eq_11}
\frac{\dm}{\dm t}\Pi(t_0,x)
=\left(\frac{\dm}{\dm t}P(v(t_0),x)\right)\cdot Rv(t_0)   +     P(v(t_0),x)\cdot R\dot{v}(t_0).
\end{equation}
However, $P(v(t_0),x)=0$ by assumption and hence by \eqref{eq_11} and \eqref{eq_a1},
\begin{equation*}
\frac{\dm}{\dm t}\Pi(t_0,x)
=\left(\frac{\dm}{\dm t}P(v(t_0),x)\right)\cdot Rv(t_0)=\left(-\psi(t_0)\, DG^{-1}(v(t_0))(\dot{v}(t_0))\right) \cdot Rv(t_0) . 
\end{equation*}
Recall that $DG^{-1}(v(t_0))(\dot{v}(t_0))$ is an element of the tangent line $T_{G^{-1}(v(t_0))}\partial B_N(0,1)$ and that $v(t_0)$, by definition of $G$, $v(t_0)$ is orthogonal to this line. Furthermore, $DG^{-1}(v(t_0))$, $Rv(t_0)$ and $\psi(t_0)$ (we have established this in the proof of Lemma~\ref{lem_transv}) are non-zero. Therefore, we conclude that $\frac{\dm}{\dm t}\Pi(t_0,x)\neq 0$. This proves (P3).
\end{proof}

Note that in order to establish property (P3)  and thus transversality (that is property (b) from Definition~\ref{def_PS}), we have only used the assumption that $N$ restricted to $\R^2\woz$ is of class $C^{2}$. The extra $\delta$ was needed in order to obtain the required regularity for our projection family. 
 
 \begin{proof}[Proof of Theorem~\ref{thm1}] By Theorem~\ref{mainlemma}, the projection family $\Pi:J\times \Omega\to\R$ defined by \eqref{def_Pi} satisfies Definition~\ref{def_PS} and
therefore we can apply Theorem~\ref{thmPS}. Note that by definition of $\Pi$, for all $t\in J$ and $A\subset \R^2$ measurable, the sets $P(v(t),A)$ and $\Pi(t,A)$ are isometric. Moreover, $v:J\to\sphere^1$ is a local isometry. Recall that isometries between metric spaces are known to preserve Hausdorff measure and dimension, and moreover, local isometries are known to preserve Hausdorff dimension. Without loss of generality, assume that $J$ is such that $v(J)=\sphere^1$. Then, Theorem~\ref{thm1} follows from Theorem~\ref{thmPS}.
\end{proof}

\begin{proof}[Proof of Theorem~\ref{thm2}] By Theorem 1.2 in \cite{HJJL2012}, from Theorem~\ref{mainlemma} it follows that: A $\Hhh^1$-measurable set $A\subset \Omega$ with $\Hhh^1(A)<\infty$ is purely $1$-unrectifiable if and only if $\Hhh^1(\Pi(t,A))=0$ for $\Lll^1$-a.e. $t\in J$. Then using the same arguments about local isometries as in the proof of Theorem~\ref{thm1}, Theorem~\ref{thm2} follows.
\end{proof}

\section{Application to $p$-norms}\label{sec_examples}
In this section we intend to apply Theorem 1.1 to the case of the $p$-norms in $\R^2$. 
We will discuss the impact of this class of examples on Theorem~\ref{thm1} in the final remarks, see Section~\ref{sec_final}.\\

For $1\leq p <\infty$, $(x,y)\in \R^2$, the $p$-norm on $\R^2$ is defined by
\begin{equation*}
\|(x,y)\|_p :=(|x|^p+|y|^p)^\frac{1}{p},
\end{equation*}
Moreover, we define 
\begin{equation*}
\|(x,y)\|_\infty  :=\max \{|x|,|y|\}.
\end{equation*}

Note that $\|\ndot\|_2$ is the Euclidean norm, and also that,
for fixed $(x,y)\in\R^2$, $\|(x,y)\|_\infty$ is the limit of $\|(x,y)\|_p$ when $p\rightarrow \infty$. Therefore we will also refer to $\|\ndot\|_\infty$ as $\|\ndot\|_p$ for $p=\infty$. 
Figure~\ref{fig_balls} illustrates the shape of the unit ball $B_{\|\ndot\|_p}(0,1)$ for different values of $p$.

\begin{figure}[h]
\begin{center}
\def\svgwidth{400pt}
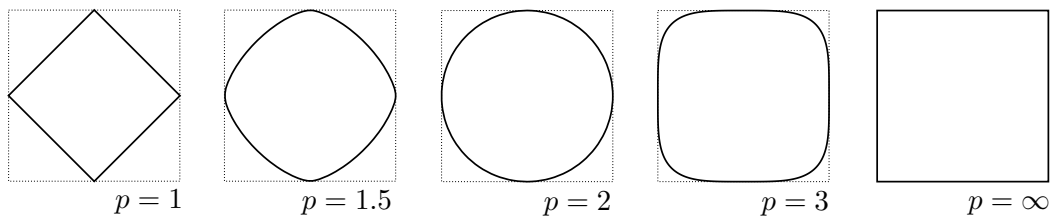
\end{center}
\caption{Shape of $B_{\|\ndot\|_p}(0,1)$ for different values of $p$ } \label{fig_balls}
\end{figure}

For $1\leq p<\infty$, $\|\ndot\|_p$ is $k$-times continuously differentiable in $\R^2\woz$ if and only if its $p$-th power, $\|\ndot\|_p^p$, is $k$-times continuously differentiable in $\R^2\woz$. Also, note that $x\mapsto |x|^p$ is $2$-times continuously differentiable in $\R$ for $p\geq 2$. Furthermore, if $1\leq p<2$, then $x\mapsto |x|^p$ is $2$-times differentiable in $\R\woz $. Recall from Section~\ref{sec_hoelder} that for all $1\leq p<\infty$, the second differential of $q\mapsto\| q\|_p^p$ (in all points where it exists)
is locally $\delta$-H\"older for some $\delta>0$. Hence, we conclude that:
\begin{itemize}[label=\raisebox{0.25ex}{\tiny$\bullet$}, itemsep=3pt]
\item For $2\leq p<\infty$: $\|\ndot\|_p$ restricted to $\R^2\woz$ is of class $C^{2,\delta}$, for some $\delta>0$.
\item For $1\leq p< 2$ : $\|\ndot\|_p$ restricted to $\R^2\backslash \{(x,y)\in\R^2: x=0 \text{ or } y=0\}$  is of class of class $C^{2,\delta}$, for some $\delta>0$. Moreover, at points in   $\{(x,y)\in\R^2: x=0 \text{ or } y=0\}\woz$, $\|\ndot\|_p$  is once but not twice differentiable.
\end{itemize}

To apply Theorem~\ref{thm1}, we distinguish the cases $p\geq 2$ and $1<p<2$. Moreover, we will explain, why Theorem~\ref{thm1} fails completely in the third case where $p\in\{1,\infty\}$:

\subsection{First case: $2\leq p<\infty$}
As explained above, in this case, the norm ${\|\ndot\|_p}$ is of class $C^{2,\delta}$ for some $\delta>0$. Also $B_{\|\ndot\|_p}(0,1)$ is strictly convex. Thus, by Proposition~\ref{prop2} and Proposition~\ref{prop3}, Theorem~\ref{thm1} applies.

\subsection{Second case: $1<p<2$}\label{sec_p<2} 

Here, $B_{\|\ndot\|_p}(0,1)$ is still strictly convex. However, as concluded above, $N$ is not twice differentiable in points that lie on the axes. So Theorem~\ref{thm1} is not applicable. Moreover, as we will show now, transversality fails for certain projection directions: First notice that by strict convexity of $B_{\|\ndot\|_p}(0,1)$ and the fact that $N$ restricted to $\R^n$ is differentiable, the Gauss map $G$ is well-defined and bijective. Thus in particular the projection formula \eqref{proj_formula} holds.
Let $q\in\R^2\woz$ be a point on the $y$-axis, $v_0\in \sphere^1$ a direction parallel to the $y$-axis and $t_0\in\R$ such that $v_0=v(t_0)$. Note that $G(v_0)=v_0$ and $q\ndot v_0=q$. Then, by either by following the intuition given in Remark~\ref{rmk_proj_intuitively} or by applying the projection formula~\eqref{proj_formula}, we see that $P(v_0,q)=0$.  On the other hand, we know that at points $\tilde{q}$ where $\partial B_{\|\ndot\|_p}(0,1)$ intersects the $y$-axis, the curvature of $\partial B_{\|\ndot\|_p}(0,1)$ equals $+\infty$. This implies that the differential of the Gauss map $DG$ tends to $+\infty$ as one approaches $v_0$, and hence, $DG^{-1}(v_0)=0$. 
Then, by \eqref{eq_a1}, it follows that $\frac{d}{dt}P(v(t_0),x)=0$. Therefore, property (P3) and by linearity~\eqref{def_Phi_linear} also transversality fail.\\
Note that $v_0$ is not the only bad projection direction (i.e. direction for which transversality fails in the above sense): By symmetry of $B_{\|\ndot\|_p}(0,1)$, we obtain the following set of bad directions:  \begin{equation}\label{def_M}
M:= \left\{ \left(
\begin{array}{c}
\cos t\\ \sin t\\
\end{array}
\right): t\in \left\{0, \tfrac{\pi}{2}, \pi, \tfrac{3\pi}{2}\right\}\right\}.
\end{equation}
As we shall see shortly, transversality as well as the necessary regularity properties hold in the complement of a neighbourhood of this set of bad directions $M$ and as a consequence, the conclusions of Theorem~\ref{thm1} can still be established. To this end, let $0<\epsilon<\tfrac{\pi}{4}$ and $A_\epsilon\subset \sphere^1$ the complement of the closed $\epsilon$-neighbourhood of $M$:  

$$A_\epsilon:=\left\{\left(\begin{array}{c}
\cos t\\ \sin t\\
\end{array}\right): t\in (\epsilon,\tfrac{\pi}{2}-\epsilon)\cup (\tfrac{\pi}{2}+\epsilon,\pi-\epsilon)\cup (\pi+\epsilon, \tfrac{3\pi}{2}-\epsilon)\cup (\tfrac{3\pi}{2}+\epsilon, 2\pi-\epsilon)\right\}.$$

Define $\tilde{A}_\epsilon:=\left\{(x,y)\in\R^2: (x,y)=rv,\, r\in\R\woz, \, v\in A_\epsilon \right\}$. Obviously, $N$ is twice differentiable in $\tilde{A}_\epsilon$ and that by the theory in Section~\ref{sec_hoelder}, $N$ restricted to $\tilde{A}_\epsilon$ is of class $C^{2\delta}$ for some $\delta>0$. 
Then, as in the proof of Theorem~\ref{thm1}, using the strict convexity of $B_N(0,1)$ and the regularity of $N$ on $\tilde{A}_\epsilon$, we can conclude that $G:A_\epsilon\to G(A_\epsilon)\subset \sphere^1$ is a $C^{1,\delta}$-diffeomorphism. 
By strict convexity of $B_N(0,1)$, the projections $P_v$ and $\Pi(v,\ndot)$ given by the equations \eqref{proj_formula} and \eqref{def_Pi}) are well-defined for $v\in A_\epsilon$, and by the regularity of $G$ on $A_\epsilon$, we can conclude that properties (P1), (P2), and (P3) from Lemma~\ref{lem_simplify_cond} hold for $\Pi:J_\epsilon\times \Omega \to \R$, where 
$$J_\epsilon:= \left\{t\in J: \left(\begin{array}{c}
\cos t\\ \sin t\\
\end{array}\right)\in A_\epsilon\right\} .$$ 
It is easy to check that, the projection formula \eqref{proj_formula} still holds in this new setting (where $G$ is non bijective $\partial B_N(0,1)\to\sphere^1$ but bijective from a suitable subset of $\partial B_N(0,1)$ onto $A_\epsilon$). Moreover, one can easily check that Lemmas~\ref{lem_simplify_cond} is still valid if we replace $J$ by $J_\epsilon$. Therefore the conclusions for of Theorem~\ref{thm1} hold for $v\in A_\epsilon$ (instead of $v\in \sphere^1$). \\[3pt]
Recall that $M$, defined in \eqref{def_M}, is a $0$-dimensional set and that, by definition of $A_\epsilon$,  $\sphere^1 \backslash M = \bigcup_{n\in\N} A_{\frac{1}{n}}.$
Thus for every measurable set $A\subset \sphere^1$, $$\dim A=\dim \left( \bigcup_{n\in\N} (A\cap A_{1/n}) \right)=\sup_{n\in\N}\left(\dim (A\cap A_{1/n}) \right).$$
Therefore, the conclusions of Theorem~\ref{thm1} hold to the full extent (i.e. for $v\in\sphere^1$).

\subsection{Third case: $p\in\{1,\infty\}$ }\label{sec_p=1}

By rotational symmetry (see e.g. Figure~\ref{fig_balls}), it suffices to consider the case $p=1$.
First notice that $B_{\|\ndot\|_1}(0,1)$ is not strictly convex and thus $P$ and $\Pi$ are not well-defined for certain directions $v\in\sphere^1$. However, this is not a main obstacle, since it is easily checked that there are only four directions $v\in\sphere^1$ for which $P$ and $\Pi$ are not well-defined: the diagonal directions. We denote the set of the four diagonal directions by $D$ and we restrict our study to directions $v\in\sphere^1\backslash D$.

\begin{figure}[h]
\begin{center}
\def\svgwidth{350pt}
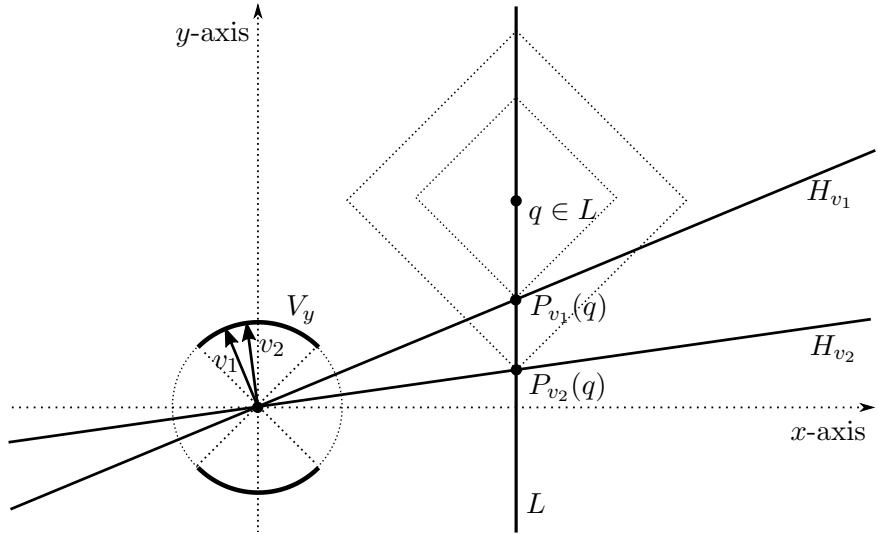
\end{center}
\caption{Projection of an arbitrary point $q$ on $L$ onto lines $H_{v_1}$ and $H_{v_2}$ orthognal to $v_1$ resp. $v_2$ in $V_y$.} \label{fig_pnorm_proj}
\end{figure}

Following the construction of the projection explained in Remark~\ref{rmk_proj_intuitively}, immediately yields that for every $v\in V_y$, where 
$V_y :=\left\{\left(\begin{array}{c}
\cos t\\ \sin t\\
\end{array}\right):t\in (\tfrac{\pi}{4},\tfrac{3\pi}{4})\cup (\tfrac{5\pi}{4},\tfrac{7\pi}{4})\right\},$ each line $L$ parallel to the $y$-axis is collapsed to one point by $P_v$, as illustrated in Figure~\ref{fig_pnorm_proj}.
Symmetrically, for every  $v\in V_x$, where $V_x:=\left\{\left(\begin{array}{c}
\cos t\\ \sin t\\
\end{array}\right):t\in (0,\tfrac{\pi}{4})\cup (\tfrac{3\pi}{4}, \tfrac{5\pi}{4})\cup (\tfrac{7\pi}{4},2\pi)\right\},$
each line parallel to the $x$-axis is collapsed to one point by $P_v$. As a straightforward consequence, the conclusions of Theorem~\ref{thm1} fail completely.

\section{Final remarks}\label{sec_final}
As we have learned, the example of $p$-norms on $\R^2$ provide an almost exhaustive range of situations that can occur: Cases where Theorem~\ref{thm1} holds to the full extent ($2\leq p<\infty$); cases where the assumptions from Theorem~\ref{thm1} are partly met and its conclusion remains true to the full extent ($1<p\leq 2$); and cases where Theorem~\ref{thm1} completely fails ($p\in\{1,\infty\}$).
Now, we consider one more example that, intuitively, lies in between the cases $1<p<2$ and $p=1$. For this example, Theorem~\ref{thm1} will be locally applicable and also the conclusions will hold only locally.\\[6pt]
To this end, we first want to recall to what extent the assumptions from Theorem~\ref{thm1} failed in the cases $1<p<2$ and $p=1$, respectively: 
In the setting where $1<p<2$, the Gauss map $G$ is defined in every point $p\in \partial B_{\|\ndot\|_p}(0,1)$. Intuitively speaking, that is $\partial B_{\|\ndot\|_p}(0,1) $ does not have corners. The only circumstance that did not fit the assumptions of Theorem~\ref{thm1} was, that $G$ fails to be locally diffeomorphic in finitely many $v\in\sphere^1$. We then cut out neighbourhoods of such bad projection directions $v$ and applied (a local version of) Theorem~\ref{thm1} to the complement of these neighbourhoods. Then, we by an easy measure theoretic argument passed to the (full version of the) conclusion of Theorem~\ref{thm1}. Thus by the procedure of Section~\ref{sec_p<2}, we can extend Theorem~\ref{thm1} to cases where $G$ fails to be a local diffeomorphism in a finite (or countable) number of points $v\in \sphere^1$.
Then, in the setting where $p=1$, the Theorem~\ref{thm1} fails completely due to the corners of $ \partial B_{\|\ndot\|_p}(0,1)$ 
(i.e. the points $q\in \partial B_{\|\ndot\|_p}(0,1)$ where $G$ is not defined): 
Each of the four corner $q\in \partial B_{\|\ndot\|_\infty}(0,1)$   yields an open set $V(q)$ of projection directions for which there exists a line $L$ so that for for all $v\in V(q)$, the projection $P_v$ collapses $L$ to a single point (compare Figure \ref{fig_pnorm_proj}). Note that by symmetry $V(q)=V(-q)$. Furthermore, the union of all these sets $V(q)$, that is $V_x\cup V_y$, covers $\sphere^1$ (up to finitely many points).\\[6pt]
We will now study the following concrete example of a (strictly convex) norm $N$ on $\R^2$ for which  $\partial B_{N}(0,1)$  has (finitely many) corners, but the sets $V$ of bad projection directions, that arise from the corners of $\partial B_{N}(0,1)$), do not exhaust $\sphere^1$:\\

Define 
\begin{equation*}
\Gamma_1=\left\{ \left(
\begin{array}{c}
\cos t\\ \sin t\\
\end{array}
\right): t\in \left( -\tfrac{\pi}{4}, \tfrac{\pi}{4}\right)\right\}, \ \ 
\Gamma_2=\left\{ \left(
\begin{array}{c}
\cos t\\ \sin t\\
\end{array}
\right): t\in \left(\tfrac{3\pi}{4},\tfrac{5\pi}{4}\right)\right\}
\end{equation*}
and define a norm $N$ on $\R^2$ by setting 
\begin{equation}\label{eq_def_norm}
\partial B_N(0,1):=\left\{ \left(
\begin{array}{c}
x\\ y\\
\end{array}
\right)+ \left(
\begin{array}{c}
\cos \tfrac{\pi}{4}\\0\\
\end{array}
\right) :  \left(
\begin{array}{c}
x\\ y\\
\end{array}
\right) \in \bar{\Gamma}_1 \right\}
\cup 
\left\{ \left(
\begin{array}{c}
x\\ y\\
\end{array}
\right)- \left(
\begin{array}{c}
\cos \tfrac{\pi}{4}\\0\\
\end{array}
\right) :  \left(
\begin{array}{c}
x\\ y\\
\end{array}
\right) \in \bar{\Gamma}_2 \right\}. 
\end{equation} 
Furthermore, set $p_1=\left(
\begin{array}{c}
0\\ 1\\
\end{array}
\right)$ and $p_2=\left(
\begin{array}{c}
0\\ -1\\
\end{array}
\right)$, see Figure~\ref{fig_example}.

\begin{figure}[h]
\begin{center}
\def\svgwidth{240pt}
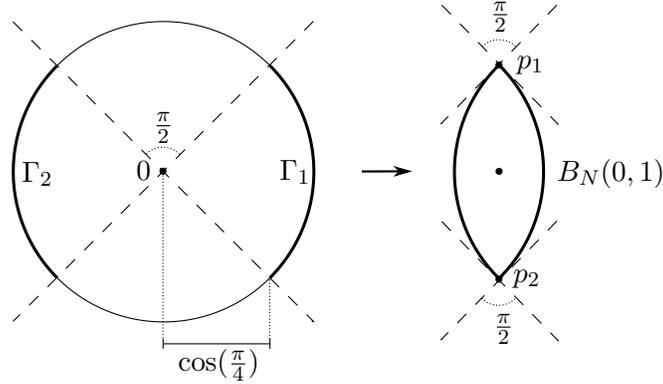
\end{center}
\caption{Construction of $\partial B_N(0,1)$ from the spherical arcs $\Gamma_1$ and $\Gamma_2$}\label{fig_example}
\end{figure}

Then, the Gauss map $G$ is well-defined and a local $C^{1,\delta}$-diffeomorphism (for some $\delta>0$) in all points in $\partial B_N(0,1)\backslash\{p_1,p_2\}$ and that the set of bad projection directions is
$$V(p_1)=V(p_2):= \left\{ \left(
\begin{array}{c}
\cos t\\ \sin t\\
\end{array}
\right): t\in \left[ \tfrac{\pi}{4},\tfrac{3\pi}{4}\right]\cup \left[ \tfrac{5\pi}{4},\tfrac{7\pi}{4}\right]\right\},$$
see Figure~\ref{fig_example_proj} (and compare Figure~\ref{fig_pnorm_proj}).  
\begin{figure}[h]
\begin{center}
\def\svgwidth{350pt}
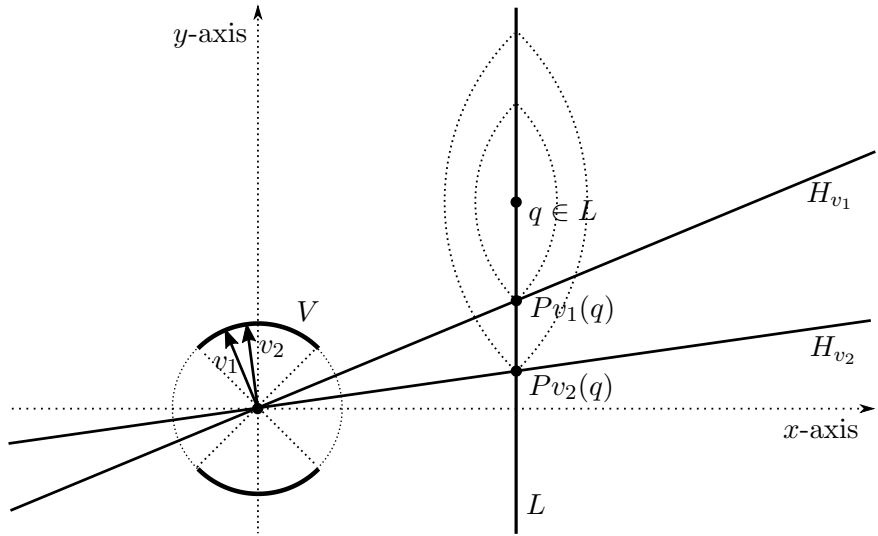
\end{center}
\caption{Projection of a point $q$ on a line $L$ parallel to the $x$-axis onto $H_{v_1}$ resp. $H_{v_2}$, where $v_1,v_2$ lie in $V:=V(p_1)=V(p_2)$.}\label{fig_example_proj}
\end{figure}

Define $E$ to be the image of $\partial B_N(0,1)$ under the Gauss map $G : \partial B_N(0,1)\backslash\{p_1,p_2\}\to\sphere^1$ and note that 
$E:=G(\partial B_N(0,1)\backslash\{p_1,p_2\} )=\sphere^1\backslash V(p_1)=\Gamma_1\cup \Gamma_2$.
Applying the arguments given in Section~\ref{sec_p<2} (case $1<p<2$), the conclusions of Theorem~\ref{thm1} hold for $E\subset \sphere^1$. \\[6pt]
As a combination of our considerations for the above example and the study of the $p$-norm for $1<p<2$, we can conclude the following local versions of our main results:

\begin{thm}\label{lem_local}
Let $N$ be a strictly convex norm on $\R^2$ and $Q=\{q_1,q_2,...\}$, $P=\{p_1,p_2...\}$ two finite or countable subsets of $\partial B_N(0,1)$, such that $\partial B_N(0,1)\backslash(P\cup Q)$ is the union of finitely resp. countably many disjoint open intervals. In addition, we assume that $N$ is such that:\begin{itemize}[label=\raisebox{0.25ex}{\tiny$\bullet$}, itemsep=2pt]
\item The Gauss map $G$ is well-defined and continuous on $\partial B_N(0,1)\backslash P$.
\item $G$ is a local $C^{2,\delta}$-diffeomorphism on $\partial B_N(0,1)\backslash (Q\cup P)$
\end{itemize}
Define $E\subset \sphere^1$ to be the image of $G: \partial B_N(0,1)\backslash P\to \sphere^1$ and for $i\in\N$, define $ V(p_i)$ to be the set of directions $v\in\sphere^1$ for which there exists an open neighbourhood $W$ in $\partial B_N(0,1)$ such that each line parallel to the vector $p_i$ is collapsed to a point by each projection $P_{\tilde{v}}$ with $\tilde{v}\in W$. Then, 
\begin{itemize}[label=\raisebox{0.25ex}{\tiny$\bullet$}, itemsep=2pt]
\item each of the sets $V(p_i)$ is the union of two intervals in $\sphere^1$ that are bijectively mapped onto each other by the antipodal map $x\mapsto x$,
\item $E=\sphere^1\backslash \bigcup_{p_i\in P}V(p_i)$,
\item $V(p_i)=V(-p_i)$, for all $p_i\in P$,
\item and, for each $\epsilon>0$ transversality holds for the projection family $\Pi:J_{P,Q,\epsilon}\times \Omega\to \R,$ where 
\begin{equation*}
 J_{P,Q,\epsilon}=\bigg\{ t\in J: v(t)\in \sphere^1\backslash \bigg(\bigcup_{p_i\in P} V(p_i)\cup Q_\epsilon\bigg)\bigg\} ,
\end{equation*} and $Q_\epsilon$ denotes the $\epsilon$-neighbourhood of the set $Q$.  Moreover, the transversality constant $C$ depends on $\epsilon$. \end{itemize}
\end{thm}

\begin{thm}\label{thm1_loc}
Under the assumptions of Lemma~\ref{lem_local}, 
Theorem~\ref{thm1} holds for $E$, that is, for all Borel sets $A\subseteq \R^2$, $s:=\dim(A)$, the following hold:
\begin{enumerate}[itemsep=3pt]
\item if $s>1$ then
 \begin{enumerate}[label=(1.\alph*)]
\item $\Hhh^{1}(P_v(A))>0$ for $\Hhh^{1}$-a.e. $v\in E$.
\item $\dim\{v\in E: \Hhh^{1}(P_v(A))=0\}\leq 2-\min\{s,1+\delta\}$.
\end{enumerate}
\item if $s\leq 1$, then \begin{enumerate}[label=(2.\alph*)]
\item $\dim(P_v(A))=\dim(A)$ for $\Hhh^{1}$-a.e. $v\in E$,
\item $\dim\{v\in E: \dim(P_v(A))<s\}\leq s$.
\end{enumerate}
\end{enumerate}
On the other hand, 
the above statements fail for all subsets  of $\, \sphere^1\backslash E$.
\end{thm}

\begin{thm}
Under the assumptions of Lemma~\ref{lem_local}: An $\Hhh^1$-measurable sets $A\subset \R^2$ with $\Hhh^1(A)<\infty$ is purely $1$-unrectifiable if and only if $\Hhh^1(P_v(A))=0$ for $\Lll^1$-a.e. $v\in E$.
\end{thm}

Note that, in the case where $N=\|\ndot\|_p$ for $1<p<2$, $P=\lm$, $Q=\{(1,0),(0,1),(-1,0),(0-1)\}$ and thus $E=\sphere^1$. On the other hand, recall that in the case where $N$ is defined by \eqref{eq_def_norm}, $Q=\lm$, $P=\{p_1,p_2\}$ as given below \eqref{eq_def_norm} and illustrated in Figure~\ref{fig_example} and $E=\Gamma_1\cup \Gamma_2$.\\[6pt]
As pointed out in the introduction, in various settings higher dimensional versions of Marstrand type theorems and Besicovitch-Federer type characterizations for purely unrectifiability are known to hold. This raises the natural question about the existence of projection theorems in higher dimensional normed spaces. This is the subject of forthcoming work in preparation. It is interesting to mention here, that transversality can only be expected for finite dimensional normed spaces, since Bate et.\,al.~\cite{Bate2017} have shown that the Besicovitch-Federer characterization of purely-unrectifiable sets (and therefore also transversality) fail to hold in the setting of infinite dimensional Banach spaces. Thus, any sort of projection theorem one might expect in such infinite dimensional spaces would have to be proven by different means, e.g. the potential theoretical methods due to Kaufman and Mattila, mentioned in the introduction. Also, one can obtain parts of Theorem~\ref{thm1} (resp. Theorem~\ref{thm1_loc}) for norms $N$ on $\R^2$ with less regularity (i.e. for $N$ restricted to $\R\woz$ of class $C^{1,1}$) by adapting these potential theoretical methods. In particular, we observe that based on the fact that the closest point projection map $(v,x)\mapsto P(v,x)\ndot Rv$ (compare \eqref{def_Pi}) is linear in $x$, we can rewrite it as $(v,x)\mapsto F(v)\ndot x$, for some mapping $F:\sphere^1\to\sphere^1$. Therefore, Marstrand type projections theorems can be obtained by studying measure and dimension preserving properties of the mapping $F$. For more details we refer to \cite{AnninaPhD}.

\bibliography{literature_projections}
\bibliographystyle{abbrv}

 \end{document}